\documentclass[11pt]{amsart}
\usepackage{tabularx,booktabs,tikz}
\usepackage{caption}
\usepackage{amsmath}
\usepackage{amsfonts}

\usepackage{amscd}
\usepackage{amsthm}
\usepackage{amssymb} 
\usepackage{latexsym}
\usepackage{eufrak}
\usepackage{euscript}
\usepackage{epsfig}
\usepackage[width=16cm]{geometry}
\usepackage{graphics}
\usepackage{array}
\usepackage{enumerate}
\usepackage{dsfont}
\usepackage{color}
\usepackage{wasysym}
\usepackage{hyperref}
\usepackage{pdfsync}

\newcommand{\Hmm}[1]{\leavevmode{\marginpar{\tiny%
			$\hbox to 0mm{\hspace*{-0.5mm}$\leftarrow$\hss}%
			\vcenter{\vrule depth 0.1mm height 0.1mm width \the\marginparwidth}%
			\hbox to
			0mm{\hss$\rightarrow$\hspace*{-0.5mm}}$\\\relax\raggedright #1}}}

\newtheorem{theorem}{Theorem}[section]
\newtheorem{lemma}[theorem]{Lemma}
\newtheorem{corollary}[theorem]{Corollary}
\newtheorem{definition}[theorem]{Definition}

\newtheorem{remark}[theorem]{Remark}

\newtheorem{proposition}[theorem]{Proposition}

\begin{document}

\title[The Existence of ground state solutions for $p$-Laplacian equations on lattice graphs]{The Existence of ground state solutions for nonlinear $p$-Laplacian equations on lattice graphs}

\author{Bobo Hua}
\address{Bobo Hua: School of Mathematical Sciences, LMNS, Fudan University, Shanghai 200433, China; ; Shanghai Center for Mathematical Sciences, Jiangwan Campus, Fudan University, No. 2005 Songhu Road, Shanghai 200438, China.}
\email{bobohua@fudan.edu.cn}

\author{Wendi Xu}
\address{Wendi Xu: School of Mathematical Sciences, Fudan University, Shanghai 200433, China}
\email{wdxu19@fudan.edu.cn}

\begin{abstract}
	In this paper, we study the nonlinear $p$-Laplacian equation
	$$-\Delta_{p} u+V(x)|u|^{p-2}u=f(x,u) $$
with positive and periodic potential $V$ on the lattice graph $\mathbb{Z}^{N}$, where $\Delta_{p}$ is the discrete $p$-Laplacian, $p \in (1,\infty)$. 
The nonlinearity $f$ is also periodic in $x$ and satisfies the growth condition 
$|f(x,u)| \leq a(1+|u|^{q-1})$  for some $ q>p$.
 We first prove the equivalence of three function spaces on $\mathbb{Z}^{N}$, which is quite different from the continuous case and allows us to remove the restriction $q>p^{*}$ in \cite{MR2768820}, where $p^{*}$ is the critical exponent for
 $ W^{1,p}(\Omega) \hookrightarrow L^{q}(\Omega)$ with $\Omega \subset \mathbb{R}^{N}$ bounded. 
  Then, using the method of Nehari \cite{MR111898,MR123775}, we prove the existence of ground state solutions to the above equation.
\end{abstract}
\par
\maketitle

\bigskip

\section{introduction}
The semilinear Schr\"{o}dinger type equation of the form
\begin{equation}\label{eq1} 
		-\Delta u + V(x)u = f(x,u), \quad 
		u\in W^{1,2}(\Omega),
\end{equation}  
where $\Omega \subset \mathbb{R}^{N}$, $V: \Omega \to \mathbb{R}$ is a given potential and $f: \Omega \times \mathbb{R} \to \mathbb{R}$ is a nonlinearity,
has been extensively studied in the literature, see, e.g., \cite{  MR709644,MR1162728,MR1163431,MR1762697,MR2005933,MR2120993,MR2271695, MR2920489} and the references therein. 
The $p$-Laplacian equation 
\begin{equation}\label{eq2} 
	-\Delta_{p} u + V(x)|u|^{p-2}u = f(x,u) 
\end{equation}  
is a nonlinear generalization of the problem \eqref{eq1} to $p \in (1,\infty)$.
In the note \cite{MR2768820}, the authors consider the zero boundary value problem of equation \eqref{eq2} with $V(x)=-\lambda > - \lambda_{1} $, where $\lambda_{1}$ is the first Dirichlet eigenvalue of the $p$-Laplacian. They obtain ground state solutions to \eqref{eq2} via the  Nehari method, which was first introduced in \cite{MR111898} to find a solution to the zero boundary value problem of \eqref{eq1} with $V(x) = \lambda \geq 0$ and $\Omega = (a,b)\subset \mathbb{R}$. Roughly speaking, the Nehari method is to minimize the corresponding functionals on  constrained manifolds and our goal is to use this method to study equation \eqref{eq2} on lattice graphs.

In recent years, analysis on graph has attracted the attention because of its importance in applications, see e.g., \cite {MR3316971, MR3523107, MR3665801,MR4053610,MR4482272,MR4581156}. 
Our work is inspired by the  following works. In \cite{MR3542963}, the authors obtain a positive solution to the Yamabe equation
\begin{displaymath}
	\left\{
	\begin{aligned}
		& -\Delta u - \alpha u = |u|^{p-2}u \quad \text{ in } \Omega^{0}\\
		&u=0 \quad \text{ on } \partial \Omega\\ 
	\end{aligned}
	\right.
\end{displaymath}
on a finite subgraph $\Omega$ in a locally finite graph via the mountain pass theorem. They also consider similar problems involving the $p$-Laplacian and poly-Laplacian by the same method. In \cite{MR3833747}, N. Zhang and L. Zhao prove via the Nehari method that the Schr\"{o}dinger type equation $$-\Delta u +(\lambda a(x)+1) u = |u|^{p-1}u $$ admits a ground state solution $u_{\lambda}$ for each $\lambda>1$ and $u_{\lambda}$ converges to the solution of a Dirichlet problem as $\lambda \to \infty$. In \cite{MR4344559}, X. Han and M. Shao generalize their work from Laplace operator to $p$-Laplace problem, $p\geq 2$, and consider more general nonlinearities. 
In our previous work \cite{MR4568177}, we prove the existence of ground state solutions of \eqref{eq1}  with  periodic or bounded potential on a lattice graph by the Nehari method. In this paper, we use this method to obtain ground state solutions of the $p$-Laplacian equation $\eqref{eq2}$ on lattice graphs.

 We first introduce the basic setting on graphs. Let $G=(\mathbb{V},\mathbb{E})$ be a simple, undirected, locally finite graph, where $\mathbb{V}$ denotes a set of vertices and $\mathbb{E}$ denotes a set of edges.
 We write $x \sim y \,( x \text{ is a neighbour of }y )$ if $ \{ x,y\}\in \mathbb{E}$. 
 The graph $G$ is undirected means that each  $ \{ x,y\}\in \mathbb{E}$ is unordered. A graph is called locally finite if each vertex has finitely many neighbours. 
We denote by $\mathbb{Z}^{N} = (\mathbb{V},\mathbb{E})$ the standard lattice graph with
$$ \mathbb{V}:= \{x=(x_{1},\cdots,x_{N}):x_{i}\in \mathbb{Z}, 1 \leq i \leq N\},$$
$$\mathbb{E}:= \Big\{ \{ x,y\}: x,y\in\mathbb{V},\, \sum_{i=1}^{N}|x_{i}-y_{i}|=1\Big\}.$$ 
Let $C(\mathbb{V})$ be the set of all functions on $\mathbb{Z}^{N}$ and $C_{c}(\mathbb{V})$ be the set of all functions with finite support.
Define the difference operator $\nabla_{xy} : C(\mathbb{V}) \to \mathbb{R}$ for any pair $\{x,y\} \in \mathbb{E}$ as 
$\nabla_{xy} u = u(y) - u(x).$ \\
For $p\in (1,\infty)$, the discrete $p$-Laplacian of $u\in C(\mathbb{V})$ is defined as
$$ \Delta_{p} u(x) :=\sum_{y\in \mathbb{V}, y \sim x}  |u(y)-u(x)|^{p-2}\big(u(y)-u(x)\big),\quad \forall x\in\mathbb{V}.$$
For $p\in [1,\infty)$, let $\mathcal{E}_{p}: C(\mathbb{V}) \to \mathbb{R}$ be the $p$-Dirichlet functional on $\mathbb{Z}^{N}$, defined as 
$$\mathcal{E}(u) := \frac{1}{2}\sum_{x,y\in \mathbb{V}, x\sim y} |u(x)-u(y)|^{p} .$$
In this paper, we adopt the definition of $\Delta_{p}$ in \cite{MR4058211}, which is different from that in \cite{MR4344559} when $p \neq 2$. This allows us to extend our results to $p\in(1,\infty)$.
Let $e_{i}$ be the unit vector in the $i$-th coordinate.
 For T $\in \mathbb{N}$, $g\in C(\mathbb{V})$ is called T-periodic if 
$$g(x+Te_{i})=g(x),\quad \quad \forall x\in\mathbb{Z}^{N},1\leq i \leq N.$$

To state our work, we introduce the following assumptions on the nonlinearity $f:\mathbb{V} \times \mathbb{R}\to \mathbb{R}$ and the potential $V\in C(\mathbb{V})$ for the equation \eqref{eq2}.\\ 
$(A_{1})$\,$f(x,u)$  is continuous with respect to $u\in \mathbb{R}$ and satisfies the growth condition: $\exists\, a>0, q>p$, s.t.  
\begin{equation}\notag
	|f(x,u)| \leq a(1+|u|^{q-1}), \quad \forall x\in\mathbb{V}.
\end{equation} 
$(A_{2})$\, $V(x)$ and $f(x,u)$ are both T-periodic in   $x$ 
and $V(x)>0$ for each vertex $x$.\\
$(A_{3})$\,$f(x,u) = o(|u|^{p-1})$ uniformly in $x$ as $u \to 0$.\\ 
$(A_{4})$\,$u\mapsto \frac{f(x,u)}{|u|^{p-1}}$ is strictly increasing on $(-\infty,0)$ and $(0,+\infty)$ respectively.\\   
$(A_{5})$\,$ \frac{F(x,u)}{|u|^p} \to \infty $ uniformly in $x$ as $|u| \to \infty $, where $F(x,u)= \int_{0}^{u} f(x,s) \,d s $.\\



Let $\ell^{p}(\mathbb{V},\mu), p\in[1,\infty],$ be the space of $\ell^{p}$ summable functions on $\mathbb{V}$ w.r.t. the counting measure $\mu$. We write $\|\cdot\|_{p}$ as the $\ell^{p}(\mathbb{V},\mu)$ norm, i.e., 
\begin{equation}\notag
	\|u\|_{p}:=
	\left\{ 	\begin{aligned}
		                  & \sum_{x\in\mathbb{V}} |u(x)|^{p},\quad 1\leq p < \infty,\\
		                  &\sup_{x\in\mathbb{V}} |u(x)|,\quad \quad  p=\infty.
               	\end{aligned}
	\right.
\end{equation}
As we know, for $1\leq p< \infty $, $\ell^{p}(\mathbb{V})$ is the completion of  $C_{c}(\mathbb{V})$ under the norm ${\| \cdot \|_{p}}$. Let
$W^{1,p}(\mathbb{V}) =\{ u\in C(\mathbb{V}) : \|u\|_{W^{1,p}(\mathbb{V})} < \infty \},$  $ 1< p <\infty,$ where
 \begin{displaymath}
 	\begin{aligned}
 		\|u\|_{W^{1,p}(\mathbb{V})} 
 		&:=\big( \mathcal{E}_{p}(u)+\|u\|_{p} \big)^{\frac{1}{p}}\\
 		&= \Big( \sum_{x,y\in \mathbb{V},x\sim y} \frac{1}{2}| \nabla_{xy}u |^{p} + \sum_{x\in \mathbb{V}}|u(x)|^{p} \Big)^{\frac{1}{p}}.
 	\end{aligned}
 \end{displaymath}
 $$$$
Moreover, from assumption $(A_{2})$, there exist
 $V_{1}, V_{2}>0$ such that $0<V_{1} \leq V(x) \leq V_{2}$ for each $x\in \mathbb{V}$. 
 We can define a new norm $\| \cdot \|$ on $C_{c}(\mathbb{V})$ by
 $$\|u\| := \Big( \sum_{x,y\in \mathbb{V},x\sim y}\frac{1}{2}| \nabla_{xy}u  |^{p} + \sum_{x\in \mathbb{V}}V(x)|u(x)|^{p} \Big)^{\frac{1}{p}}$$
and set 
$$E =\{ u\in C(\mathbb{V}) : \|u\| < \infty \}.$$
The variational functional corresponds to the equation $\eqref{eq2}$ is
\begin{displaymath}\label{functional}
	\begin{aligned}
		\Phi(u) &= \frac{1}{p} \big(\sum_{x,y\in \mathbb{V},x\sim y}  \frac{1}{2}| \nabla_{xy}u  |^{p} +\sum_{ x\in\mathbb{V}}V(x)|u(x)|^{p}\big)  - \sum_{x\in \mathbb{V}} F(x,u)\\
		&=  \frac{1}{p} \|u\|^{p} - \sum_{x\in \mathbb{V}} F(x,u).
	\end{aligned}
\end{displaymath}
One can check that $\Phi \in C^{1}(E,\mathbb{R})$ with the derivative given by
$$
\Phi'(u)v = \sum_{x,y\in \mathbb{V},x\sim y} \frac{1}{2}| \nabla_{xy}u |^{p-2} (\nabla_{xy}u) (\nabla_{xy}v) +\sum_{ x\in \mathbb{V}}V(x)|u(x)|^{p-2}uv - \sum_{x\in \mathbb{V}} f(x,u)v
,\quad \forall v\in E.
$$
In particular,
$$
\Phi'(u)u = \|u\|^{p}- \sum_{x\in \mathbb{V}} f(x,u)u.
$$
Set
$$\mathcal{N} = \{ u \in E\backslash \{0\} : \Phi'(u)u = 0 \}$$
and$$c:=\inf_{u \in \mathcal{N}} \Phi(u).$$ 
$\mathcal{N}$  is called the Nehari manifold, which contains all nontrivial critical points of $\Phi$. We say that $u$ is a ground state solution if it possesses the least positive energy among all  solutions for $\eqref{eq2}$.
The following is our main result.

 \begin{theorem}\label{thm1}
  Assume that $f$ and $V$ satisfy $(A_{1})$-$(A_{5})$. 
 	Then $c>0$, $c$ is attainable and each $u\in \mathcal{N}$ with $\Phi(u)=c$ is a ground state solution for $\eqref{eq2}$.
 \end{theorem}

This paper is organized as follows. In Section 2, we introduce some basic properties of function spaces defined on lattice graphs. In Section 3, we construct the framework of the Nehari method, and the proof of Theorem \ref{thm1} is contained in Section 4.

\section{Preliminaries}
\begin{lemma}\label{lem2.1}
 $\| \cdot \|_{p}$, $\| \cdot \|_{W^{1,p}(\mathbb{V})}$ and $\|  \cdot \|$ are equivalent norms and $\ell^{p}(\mathbb{V})= W^{1,p}(\mathbb{V})=E $.
\end{lemma}

\begin{proof}
	On one hand, $0<V_{1} \leq V(x) \leq V_{2}$ for each $x\in \mathbb{V}$ implies 
	$$\nu_{1}\|u\|_{W^{1,p}} \leq \|u\| \leq \nu_{2}\|u\|_{W^{1,p}}, \quad  \forall u \in C_{c}(\mathbb{V}),$$
	where $\nu_{1}= \text{min}\{V_{1}^{\frac{1}{p}},1\}$, $\nu_{2}=\text{max}\{V_{2}^{\frac{1}{p}},1 \}$.
	On the other hand, 
	$$	\mathcal{E}_{p}(u)
	=\sum_{x,y\in \mathbb{V},y \sim x} \frac{1}{2} |u(y)-u(x)|^{p} 
	\leq C(p,N)\|u\|_{p}^{p}. $$     
Hence, $$\|u\|_{p} \leq \|u\|_{W^{1,p}} \leq \big( 1+C(p,N) \big)^{\frac{1}{p}}\|u\|_{p}.$$
This yields that 
$\| \cdot \|_{p} \sim  \| \cdot \|_{W^{1,p}(\mathbb{V})} \sim \|  \cdot \|$. 
\end{proof}
 
 \begin{lemma}\label{reflexive}
 	$E$ is a reflexive Banach space.
 \end{lemma}
\begin{proof}
   Since $\ell^{p}(\mathbb{V})$ is a uniformly convex Banach space, it is reflexive.
   The reflexive space remains reflexive for an equivalent norm, so both $W^{1,p}(V)$ and $E$ are reflexive as well.
\end{proof}

\begin{lemma}[Integration by parts of $p$-Laplacian]
	For any $u \in E, v\in C_{c}(V),$
	$$
\frac{1}{2}	\sum_{x,y\in\mathbb{V}, x\sim y} |\nabla_{xy} u|^{p-2}\nabla_{xy}u\nabla_{xy}v = \sum_{x\in \mathbb{V}} (-\Delta_{p}u)v.
	$$
\end{lemma}
	The proof is similar to  that of Green's formula in \cite{MR3822363}.
Based on this lemma, we are ready to define the weak solution of \eqref{eq2}.
\begin{definition}

Let $u\in E$. We say that $u$ is a weak solution of \eqref{eq2} if
   $$
	\frac{1}{2}\sum_{x,y\in\mathbb{V},x\sim y} |\nabla_{xy} u|^{p-2}\nabla_{xy} u \nabla_{xy} v + \sum_{ x\in \mathbb{V}}V(x)|u|^{p-2}uv
	= \sum_{x\in\mathbb{V}} f(x,u)v, \quad \forall v\in E.
	$$
\end{definition}
For any given $y \in \mathbb{V}$, by using integration by parts and  taking  the test function 
$v(x) = \delta_{y}(x)$, we obtain the following lemma.
\begin{lemma}
	If $u$ is a weak solution of (\ref{eq2}),  then $u$ is a pointwise solution.  
\end{lemma}

\begin{lemma}\label{embedding}
	Let $\mathbb{Z}^{N} = (\mathbb{V}, \mathbb{E})$ be a lattice graph. Then,\\
	$(i)$ $\ell^{p}(\mathbb{V})$ is embedded into $\ell^{q}(\mathbb{V})$ for any $1 \leq p < q \leq \infty$.\\
   $(ii)$ For any bounded sequence $\{u_{n}\}_{n} \subset E$, there exists $u\in E$ such that, up to a subsequence,
	\begin{equation}\notag
		\left\{ 	\begin{aligned}
			& u_{n} \rightharpoonup u \quad \quad \quad \quad \text{in} \,\, E,\\
			& u_{n}(x) \to u(x) \quad \forall x\in \mathbb{V}.
		\end{aligned}
		\right.
	\end{equation}
\end{lemma}

\begin{proof}
	(i) is well-known. \\
	(ii) By Lemma \ref{reflexive}, $E$ is a reflexive Banach space. Since $\{u_{n}\}_{n}$ is bounded in $E$, we have $u_{n}\rightharpoonup u$ in $E$ after passing to a subsequence. Moreover, $\{ u_{n}\}_{n}$ is also bounded in $\ell^{p}(\mathbb{V})$, $ 1< p < \infty$, and therefore, $u_{n} \rightharpoonup u$ in $\ell^{p}(\mathbb{V})$. 
	It follows that,
	$$\lim_{n\to \infty} \sum_{x\in\mathbb{V}} u_{n}\phi  = \sum_{x\in\mathbb{V}} u\phi , \quad \forall \phi \in \ell^{p'}(\mathbb{V}), \,p' := \tfrac{p}{p-1}.
	$$
	For  each vertex $x_{0} \in \mathbb{V}$, taking $\phi = \delta_{x_{0}}$, we have 
	$\lim_{n} u_{n}(x_{0}) = u(x_{0})$.
\end{proof}
At the end of this section, we give two lemmas which will be used in the subsequent proofs.
\begin{lemma}\label{epsilon}
  Suppose $f$ satisfies $(A_{1})$ and $(A_{3})$. Then, for any $\epsilon>0$, there exists $C_{\epsilon}>0$ such that 
   $$
   |f(x,u)| \leq \epsilon|u|^{p-1} + C_{\epsilon} |u|^{q-1}, \quad \forall u\in \mathbb{R}, x\in \mathbb{V}.
   $$
\end{lemma}
  
 \begin{proof}
 	By $(A_{3})$, for any given $\epsilon >0$, there exists $\delta >0$ such that, $\forall |u|<\delta$,
 	\begin{displaymath}
 		|f(x,u)| \leq \epsilon |u|^{p-1}, \quad \forall  x\in \mathbb{V}.
 	\end{displaymath}
 	When $|u| \geq \delta$, choose  $C_{\delta} \geq \frac{a}{\delta^{q-1}}$. 
 	Substitute this into $(A_{1})$, we have 
    \begin{displaymath}
    	|f(x,u)| \leq a(1+|u|^{q-1}) \leq (C_{\delta}+a)|u|^{q-1} =: C_{\epsilon} |u|^{q-1}.
    \end{displaymath}
 \end{proof}

\begin{lemma}\label{f&F}
	 Suppose $f$ satisfies $(A_{3})$-$(A_{4})$. Then, 
	 $$0 < F(x,u) < \frac{1}{p} f(x,u)u, \quad \forall u \in \mathbb{R}\backslash \{0\}.$$
\end{lemma}

\begin{proof}
	Combining $(A_{3})$ with $(A_{4})$, we have $\frac{f(x,u)}{|u|^{p-1}}<0$ if $u<0$ and $\frac{f(x,u)}{|u|^{p-1}}>0$ if $u>0$, which gives $f(x,u)u>0$, $\forall u \neq 0$.
	 Hence, $F(x,u)=\int_{0}^{u} f(x,t) dt >0$.
	 If $u>0$, by $(A_{4})$,
	 $$	F(x,u) =\int_{0}^{u} \frac{f(x,t)}{t^{p-1}}t^{p-1} dt 
	 < \frac{f(x,u)}{u^{p-1}}\int_{0}^{u} t^{p-1} dt
	 = \frac{1}{p}f(x,u)u. $$
	 If $u<0$, the proof is similar.
\end{proof}

\section{Existence of ground state solutions}
In this section, we always assume that $(A_{1})$-$(A_{5})$ are satisfied. We prove the existence of ground state solutions for the equation \eqref{eq2} by the Nehari method. Recall that Nehari manifold is defined as 
$$ \mathcal{N} = \{ u \in E\backslash \{0\} : \Phi'(u)u =0  \},$$ 
which contains all nontrivial critical points of $\Phi$.  Denote the unit sphere in $E$ by $$S=\{ u\in E : \|u\|=1 \}.$$
In the following, we first show that $\mathcal{N}$ is homeomorphic to $S$. And then, the problem of looking for critical points of $\Phi$ on $\mathcal{N}$ is reduced to seeking critical points of a new functional $\Psi$ on $S$.

\begin{proposition}	\label{3.1 }
	Set $\alpha_{w}(s) = \Phi(sw)$  for $w\in E\backslash\{0\}$, $0<s<\infty$. Then we have the following.\\
	$(i)$ For each $w\in E\backslash\{0\}$, $\alpha_{w}$ admits a unique maximum point $s_{w}\in(0,\infty)$. \\
	$(ii)$ There exists $\delta >0$ such that $s_{w} \geq \delta$ for all $w$ in $S$.\\
	$(iii)$ For each compact set $K\subset S$, there exists $C_{K}>0$ such that $s_{w} \leq C_{K}$ for all $w$ in $K$.
\end{proposition}

\begin{proof}
	$(i)$ Let $w\in E\backslash \{0\}$. By calculation,
		\begin{displaymath}
		\begin{aligned}	
			\alpha_{w}'(s) &= \Phi'(sw)w = s^{p-1}\|w\|^{p} - \sum_{x\in \mathbb{V}} f(x,sw(x))w(x)\\
			&=s^{p-1} \Big( \|w\|^{p}- \sum_{x\in\mathbb{V}} \frac{|f(x,sw)|}{|sw|^{p-1}}|w|^{p}  \Big)\\
			&=:s^{p-1}\theta_{w}(s).
		\end{aligned}	
	\end{displaymath}
 According to $(A_{4})$, $\theta_{w}(s)$ is strictly decreasing on $(0,\infty)$.  Moreover, we know from $(A_{3})$ that  $\theta_{w}(s)>0$ for sufficiently small $s>0$. Lemma \ref*{f&F} and $(A_{5})$ implies
$$  \theta_{w}(s)<\|w\|^{p} - p\sum_{x\in \mathbb{V}} \frac{F(x,sw)}{|sw|^{p}} |w|^{p} 
 \to-\infty \quad  \text{ as } s \to +\infty.$$
Therefore, there exists unique $s_{w}\in (0,\infty)$ such that 
\begin{equation}\label{zero}
      \alpha_{w}'(s_{w}) 	= \Phi'(s_{w}w)w = 0
\end{equation}
and $\alpha_{w}(s)$ increases on $(0,s_{w})$ and decreases on $(s_{w},\infty)$.\\
$(ii)$ Let $w \in S$ and
$$\theta_{w}(s) = 1 - \frac{1}{s^{p-1}} \sum_{x\in \mathbb{V}} |f(x,sw)||w|. $$
By Lemma \ref{epsilon}, for any $\epsilon>0$, there exists $C_{\epsilon}>0$ such that 
\begin{displaymath}
	\begin{aligned}
	   \theta_{w}(s) 
	   &\geq  1 - \frac{1}{s^{p-1}} \sum_{x\in \mathbb{V}} \Big( \epsilon|sw|^{p-1} +C_{\epsilon}|sw|^{q-1} \Big)|w|  \\
	   	   &=  1 - \sum_{x\in \mathbb{V}} \Big( \epsilon|w|^{p} +C_{\epsilon}s^{q-p}|w|^{q} \Big) \\
	   	   &=1-\big( \epsilon\|w\|_{p}^{p} + C_{\epsilon} s^{q-p}\|w\|_{q}^{q}  \big).
	\end{aligned}
\end{displaymath}
One easily sees that 
$\|w\|_{q} \leq \|w\|_{\infty}^{\frac{q-p}{q}}\|w\|_{p}^{\frac{p}{q}} \leq \|u\|_{p}$. By Lemma \ref{lem2.1}, $\|w\|_{p} \leq V_{1}^{-\frac{1}{p}}$ as $\|w\|=1$. 
It follows that 
 $$\theta_{w}(s) \geq \big(1- \epsilon V_{1}^{-1} \big) - C_{\epsilon} s^{q-p} V_{1}^{-\frac{q}{p}}.$$
Take $\epsilon =\epsilon_{_0} \leq \frac{V_{1}}{2}$ and $\delta > 0$ small enough such that
$$\theta_{w}(s) \geq \frac{1}{2} - C_{\epsilon_{_0}}s^{q-p}V_{1}^{-\frac{q}{p}} >0, \quad \forall s\in (0,\delta).$$
Therefore,
 $$\alpha_{w}'(s) = s^{p-1}\theta_{w}(s) >0,\quad \forall  s\in (0,\delta),$$ 
 and then $s_{w} \geq \delta >0$.\\
$(iii)$
Since $\alpha_{w}'(s)<0$ on $(s_{w},\infty)$ and $\alpha'_{w}(s)>0$ on $(0,s_{w})$, it suffices to show the existence of $C_{K}$ which guarantees
$$ \alpha_{w}'(s) \leq 0  \text{ on } [C_{K}, \infty),\quad  \forall w\in K.$$
We show this by contradiction. Suppose $\{w_{n}\}\subset K $ and $s_{n}:=s_{w_{n}} \to \infty$ as $n \to \infty$. Note that
\begin{equation}\label{tri}
	\begin{aligned}
		0<\theta_{w_{n}}(s_{n}) 
	= &1 - \sum_{x\in \mathbb{V}} \frac{f(x,s_{n}w_{n})w_{n}}{s_{n}^{p-1}}  \\
	\leq &1-p\sum_{x\in \mathbb{V}} \frac{F(x,s_{n}w_{n})}{s_{n}^{p}}. 
	\end{aligned}
\end{equation}
Since $K$ is compact, we may assume that $w_{n} \to w\in S$ and  $w_{n}(x) \to w(x)$ as $n\to \infty$ for each vertex $x$. 
For $ x_{0}\in \mathbb{V}$ 
with $w_{n}(x_{0}) \to w(x_{0})\neq 0$, by $(A_{5})$,
\begin{equation}\notag
	\begin{aligned}
	\sum_{x \in \mathbb{V}} \frac{F(x,s_{n}w_{n})}{s_{n}^{p}}
	&= \sum_{x\in \mathbb{V}} \frac{F(x,s_{n}w_{n})}{|s_{n}w_{n}|^{p}}|w_{n}(x)|^{p} \\
	&\geq \frac{F(x_{0},s_{n}w_{n}(x_{0}))}{|s_{n}w_{n}(x_{0})|^{p}}|w_{n}(x_{0})|^{p}
	\to +\infty \quad \text{ as } n\to \infty.
	\end{aligned}
\end{equation}
This yields a contraction to $(\ref{tri})$. 
\end{proof}
\begin{remark}\label{rm3.2}
   In fact, \eqref{zero} implies $ s_{w}w \in \mathcal{N}$. One can see from Lemma \ref{3.1 }(i) that, for each $w\in E\backslash \{0\}$, the ray $s\mapsto sw$ intersects with $\mathcal{N}$ at precisely one point $s_{w}w$.
\end{remark}

\begin{proposition}\label{3.2}
	Set  $\hat{m}: E\backslash\{0\} \to \mathcal{N}$,\, $\hat{m}(w):=s_{w}w$ and $m=\hat{m}|_{S}$. Then $\hat{m}$ is continuous and $m: S \to \mathcal{N}$ is a homeomorphism.
\end{proposition}
\begin{proof}
	Let $u\in E\backslash \{0\}$. It suffices to show that each sequence $\{u_{n}\}_{n} \subset E \backslash \{0\}$ with $u_{n} \to u\, (u\neq 0)$ admits a subsequence such that $\hat{m}(u_{n}) \to \hat{m}(u)$. Since $\hat{m}(tu)=\hat{m}(u)$ for any $t>0$, we may assume that $\|u_{n}\|=\|u\|=1$. Let $\hat{m}(u_{n})=s_{n}u_{n}$ and $\hat{m}(u)=s_{u}u$. By Proposition \ref{3.1 }, there exist $\delta>0, C>0$ such that $\delta \leq s_{n} \leq C$, $\forall n \in \mathbb{N}$. We may assume that $s_{n}\to \bar{s}>0$ after passing to a subsequence.
	Then $\hat{m}(u_{n}) = s_{n}u_{n} \to \bar{s}u$ as $n\to \infty$ and it suffices to show $\bar{s} u = s_{u}u$. 
	According to Proposition \ref{3.1 }(i) and Remark \ref{rm3.2}, $\Phi(s_{u}u) \geq \Phi(\bar{s}u)$ as $s_{u}u \in \mathcal{N}$. The reverse inequality follows from 
	$$\Phi(s_{u}u) = \lim_{n} \Phi(s_{u}u_{n}) \leq \lim_{n}\Phi(s_{n}u_{n}) = \Phi(\bar{s}u)$$
	since $\Phi$ is continuous on $E$. 
In fact,
	$$\Phi(w_{n})-\Phi(w) = \frac{1}{p}\big( \|w_{n}\|^{p}- \|w\|^{p} \big) - \sum_{x\in\mathbb{V}} \big( F(x,w_{n}) - F(x,w) \big),$$
	and
	 $$| \sum_{x\in \mathbb{V}} \big(F(x,w_{n}) - F(x,w)\big)  | 
	\leq \sum_{x\in \mathbb{V}} |f(x,\zeta_{n})||w_{n} - w|, $$
where $\zeta_{n}$ is between $w_{n}$ and $w$. Take $\epsilon = 1 $ in Lemma \ref{epsilon}, we have
\begin{equation}
     |f(x,\zeta_{n})| \leq |\zeta_{n}|^{p-1} + C_{1}|\zeta_{n}|^{q-1}.
\end{equation} 
Then, by H\"{o}lder inequality, 
\begin{equation}
	\begin{aligned}
		   &  | \sum_{x\in\mathbb{V}} F(x,w_{n}) - F(x,w)  | \\	
		\leq & \sum_{x\in \mathbb{V}} |w_{n} - w|( |\zeta_{n}|^{p-1} + C_{1}|\zeta_{n}|^{q-1} )  \\
		\leq& \big(\sum_{x\in\mathbb{V}} |w_{n} - w|^{p} \big)^{\frac{1}{p}} \big(\sum_{x\in\mathbb{V}} |\zeta_{n}|^{p}  \big)^{\frac{p-1}{p}}
		+ C_{1}\big(\sum_{x\in\mathbb{V}} |w_{n} - w|^{q} \big)^{\frac{1}{q}} \big(\sum_{x\in\mathbb{V}} |\zeta_{n}|^{q}  \big)^{\frac{q-1}{q}}\\
		= &\|w_{n}-w\|_{p} \|\zeta_{n}\|_{p}^{p-1}+
		C_{1}\|w_{n}-w\|_{q} \|\zeta_{n}\|_{q}^{q-1}\\
		\to & 0 \quad \text{ as } n \to \infty.
	\end{aligned}
\end{equation}
This yields the continuity of $\hat{m}$.
Now we consider the restriction of $\hat{m}$ on the unit sphere $m: S \to \mathcal{N}$.
Let $w_{i}\in S$, $m(w_{i})=s_{i}w_{i}, i=1,2$. If $m(w_{1}) = m(w_{2})$, i.e., $s_{1}w_{1}=s_{2}w_{2}$, we have $s_{1} = s_{2}$ after taking norm on both sides.
Since $s_{i} \geq \delta >0$, one sees $w_{1}=w_{2}$ and $m$ is injective. 
For each $u\in\mathcal{N}$, we know $\frac{u}{\|u\|} \in S$ and $m(\frac{u}{\|u\|})=u$ by Proposition \ref{3.1 }. Therefore $m$ is surjective and thus homeomorphism with its continuous inverse given by $u \mapsto \frac{u}{\|u\|}$, proving the result.
\end{proof}
Consider the composite functional 	$\hat{\Psi} := \Phi \circ \hat{m} $
\begin{eqnarray}
	\hat{\Psi} : E\backslash \{0\} &\to& \mathbb{R}  \nonumber \\
                                                        	w&\mapsto& \hat{\Psi}(w) = \Phi(\hat{m}(w))  \nonumber
\end{eqnarray}
The goal is to establish a one-to-one correspondence between the critical points of $\hat{\Psi}|_{S}$ and $\Phi|_{\mathcal{N}}$.

\begin{proposition}\label{prop3.4}
	$\hat{\Psi} \in C^{1}(E\backslash\{0\},\mathbb{R})$ and 
	$$\hat{\Psi}'(w)z = \frac{\|\hat{m}(w)\|}{\|w\|} \Phi'(\hat{m}(w))z,\quad \forall w,z \in E\backslash\{0\}.$$
\end{proposition}

\begin{proof}
Given $w,z \in E\backslash \{0\}$, take $\gamma>0$ small enough such that $w_{t}:=w+tz \neq 0, \forall |t|<\gamma$.  
Let $\hat{m}(w_{t}) = s_{t}(w_{t})$.
Then $\hat{m}(w) = s_{0}w$ and $s_{t}w_{t}\to s_{0}w$ as $t \to 0$ from Proposition \ref{3.2}.
 We calculate that
\begin{equation} \label{5}
	\begin{aligned}
	   & \frac{1}{t}(\hat{\Psi}(w_{t})-\hat{\Psi}(w)) 
	  = \frac{1}{t} (\Phi(s_{t}w_{t})-\Phi(s_{0}w)) 
  \geq  \frac{1}{t} (\Phi(s_{0}w_{t})-\Phi(s_{0}w))   \\
	=& \frac{1}{t} \Phi'\big(s_{0}[w+\eta_{t}(w_{t}-w)]\big)  s_{0}(w_{t}-w),   \quad \eta_{t} \in (0,1),\\
	\to & \Phi'(s_{0}w)s_{0}z \quad \text{ as } t\to 0,
	\end{aligned}
\end{equation}
and
\begin{equation}\label{6}
	\begin{aligned}
		& \frac{1}{t}(\hat{\Psi}(w_{t})-\hat{\Psi}(w)) 
		\leq  \frac{1}{t} (\Phi(s_{t}w_{t})-\Phi(s_{t}w))   \\
		=& \frac{1}{t} \Phi'\big(s_{t}[w+\gamma_{t}(w_{t}-w)]\big)  s_{t}(w_{t}-w)  ,   \quad \gamma_{t} \in (0,1),\\
		\to & \Phi'(s_{0}w)s_{0}z \quad \text{ as } t\to 0.
	\end{aligned}
\end{equation}

Combining \eqref{5} with \eqref{6}, we get
\begin{equation}\nonumber
	\hat{\Psi}'(w)z = \lim_{t\to 0} \frac{\hat{\Psi}(w_{t})-\hat{\Psi}(w)}{t}= \Phi'(s_{0}w)s_{0}z
	= \frac{\|\hat{m}(w)\| }{\|w\|} \Phi'(\hat{m}(w))z.
\end{equation}
Since $\hat{m}$ is continuous and $\Phi \in C^{1}(E\backslash \{0\}, \mathbb{R})$, $\hat{\Psi}$ admits comtinuous Gateaux derivative on $E\backslash\{0\}$. Thus, by Proposition 1.3  in \cite{MR1400007}, $\hat{\Psi} \in C^{1} (E\backslash \{0\},\mathbb{R})$.
\end{proof}
One easily gets the following lemma.
\begin{lemma}
	Let $$\phi(u) = \frac{1}{p}\|u\|^{p}
	= \sum_{x,y\in \mathbb{V},x\sim y}\frac{1}{2}| \nabla_{xy}u |^{p} + \sum_{x\in \mathbb{V}}V(x)|u(x)|^{p}.$$
	 Then $\phi \in C^{1}(E,\mathbb{R})$ and
	$$\phi(w) = \frac{1}{p}, \quad \phi'(w)w = 1,\quad \quad \forall w\in S, $$
where
	$$\phi'(w)v =\sum_{x,y\in \mathbb{V},x\sim y} \frac{1}{2}| \nabla_{xy}w |^{p-2} (\nabla_{xy}w) (\nabla_{xy}v)+ \sum_{x\in \mathbb{V}}V(x)|w(x)|^{p-2}wv ,\quad \forall v\in E.$$
	For each $w\in S$, set 
	$$T_{w}(S) = \{ v\in E: \phi'(w)v = 0 \} = ker(\phi'(w)).$$
	We have direct sum decomposition:
	$$E=\mathbb{R}_{w} \oplus T_{w}S.$$
\end{lemma}
 By Proposition \ref{prop3.4}, we have the following corollary.
\begin{corollary}\label{corollary3.5}
	 Set $\Psi = \hat{\Psi}|_{S}$. Then $\Psi \in C^{1}(S)$ and
	 $$
	 \Psi'(w)z = \|m(w)\| \Phi'(m(w))z,\quad \forall z\in T_{w}S,
	 $$
\end{corollary}
where $\Psi'$ is understood as the derivative of a function on Banach manifold.
\begin{proposition}\label{coercive}
	$\Phi$ is coercive on $\mathcal{N}$, i.e., if $\{u_{n}\} \subset \mathcal{N}$ with $\|u_{n}\| \to \infty$, then $\Phi(u_{n}) \to \infty$ as $n \to \infty$.
\end{proposition}

\begin{proof}
Argue by contradiction. Suppose  $\{u_{n}\} \subset \mathcal{N}$ with $\|u_{n}\| \to \infty$ as $n \to \infty$ and $\Phi(u_{n}) \leq d$ for each $n\in \mathbb{N}$.
 Let $v_{n} = \frac{u_{n}}{\|u_{n}\|} \in S,\, \forall n\in \mathbb{N}$.
 Note that $\sup_{n\in \mathbb{N}} \|v_{n}\|_{p} < \infty$.
Suppose that $v_{n} \to 0$ in $\ell^{q}(\mathbb{V})$. By Lemma \ref{epsilon}, for any $\epsilon >0$, there exists $C_{\epsilon}>0$ such that 
$$ \big| \sum_{x\in\mathbb{V}} F(x,\lambda v_{n}) \big|
\leq \frac{1}{p} \big(\epsilon \|\lambda v_{n}\|_{p}^{p} + C_{\epsilon}\| \lambda v_{n}\|_{q}^{q}\big) 
\leq C\epsilon + o(1), \quad n\to \infty,
$$
for any $\lambda>0$.
Letting $\epsilon \to 0^{+}$, we have 
$$\sum_{x\in\mathbb{V}} F(x,\lambda v_{n})	\to 0,\quad \forall \lambda>0.$$
Then Proposition \ref{3.1 } gives
\begin{equation}
	d \geq \Phi(u_{n}) \geq \Phi(\lambda v_{n}) 
	= \frac{1}{p} \lambda^{p} - \sum_{x\in\mathbb{V}} F(x,\lambda v_{n})  
	\to \frac{1}{p} \lambda^{p} \text{ as } n\to \infty, \quad \forall \lambda>0.
\end{equation}
This is a contradiction. Hence, $v_{n} \nrightarrow 0$ in $\ell^{q}(\mathbb{V})$, i.e., 
$$\varlimsup_{n\to \infty} \|v_{n}\|_{q} =: A >0.$$
	By the interpolation inequality and $q>p>1$,
	$$\|v_{n}\|_{q} \leq \|v_{n}\|_{p}^{\frac{p}{q}}\|v_{n}\|_{\infty}^{\frac{q-p}{q}}
	\leq V_{1}^{-\frac{p}{q}}\|v_{n}\|_{\infty}^{\frac{q-p}{q}}.$$
	Hence, $$ \varlimsup_{n\to \infty} \|v_{n}\|_{\infty}>0.$$
Therefore, there exist  subsequences $\{v_{n}\} \subset S$, $\{y_{n}\} \subset \mathbb{Z}^{N}$ and $A'>0$ such that
 $$|v_{n}(y_{n})| \geq \frac{1}{2} \|v_{n}\|_{\infty} \geq A' >0,\quad \forall n\in \mathbb{N}.$$
Set
$\tilde{v}_{n}(\cdot) = v_{n} (\cdot + k_{n}T)$ with $k_{n} = (k_{n}^{1},\cdots,k_{n}^{N})\in \mathbb{V}$ such that $\{ y_{n}-k_{n}T \} \subset  \Omega$, where $\Omega = [0,T)^{N}\cap \mathbb{V}$ is a bounded set. Then
$$\|\tilde{v}_{n}\|_{\ell^{\infty}(\Omega)} \geq |v_{n}(y_{n})| \geq A' >0.$$
Let $\tilde{u}_{n}(\cdot) = u_{n} (\cdot + k_{n}T)$. By the translation invariance of $\| \cdot \|$ and $\Phi$,
$ \|\tilde{u}_{n}\| = \|u\| \to \infty, \,\tilde{u}_{n} \in \mathcal{N} \text{ and } m(\tilde{v}_{n}) = \tilde{u}_{n}.$
Since $\{\tilde{v}_{n}\} \subset S $, we may assume up to a subsequence that $\tilde{v}_{n} \rightharpoonup v$ in $E$ and $\tilde{v}_{n}(x) \to v(x)$ for each vertex in $ \mathbb{V}$. Since there are finitely many vertices in $\Omega$, we have $\tilde{v}_{n}(x_{0}) \to v(x_{0}) >0$ for some $x_{0} \in \Omega$.
Then $$|\tilde{u}_{n}(x_{0})|=\|\tilde{u}_{n}\| |\tilde{v}_{n}(x_{0})| \to \infty \text{ as }  n\to \infty,$$
where
 $\tilde{u}_{n} = m(\tilde{v}_{n}) \in \mathcal{N}.$
Note that by $(A_{5})$,
$$  \frac{F(x_{0},\tilde{u}_{n}(x_{0}))}{|\tilde{u}_{n}|^{p}(x_{0})} \to \infty  \quad \text{ as } n\to \infty.  $$
Then
\begin{equation}
		0< \frac{\Phi(\tilde{u}_{n})}{\|\tilde{u}_{n}\|^{p}} 
		= \frac{1}{p} - \sum_{x\in\mathbb{V}} \frac{F(x,\tilde{u}_{n})}{\| \tilde{u}_{n} \|^{p}} 
		\leq \frac{1}{p} - \frac{F(x_{0}, \tilde{u}_{n}(x_{0}))}{| \tilde{u}_{n} |^{p}(x_{0})}|\tilde{v}_{n}|^{p}(x_{0})  
		\to -\infty \quad 
\end{equation}
as  $n\to \infty$, which is a contradiction. This proves the proposition.
\end{proof}

\begin{proposition}\label{3.7}
	$\{w_{n}\}_{n} \subset S$  is a Palais-Smale sequence for $\Psi$ if and only if $\{u_{n}\} \subset \mathcal{N}$ is a  Palais-Smale sequence for $\Phi$, where $u_{n} = m(w_{n})$ for each $n\in\mathbb{N}$.
\end{proposition}

\begin{proof}
	Let $\{w_{n}\} \subset S$ with
	 $ \sup_{n} |\Psi(w_{n})| = \sup_{n} |\Phi(u_{n})| <\infty$, 
	 where $u_{n} = m(w_{n}) \in \mathcal{N}$ for each $n\in \mathbb{N}$.
 On one hand, by Corollary \ref{corollary3.5}, 
	 \begin{equation}\nonumber
	\|\Psi'(w_{n})\|_{S} := \sup_{
		\tiny 	\begin{array}{c}
			z  \in T_{w_{n}}S   \\
			\|z\| =1
	\end{array}          }  \Psi'(w_{n})z
	= \|u_{n}\| \sup_{
		\tiny 	\begin{array}{c}
			z  \in T_{w_{n}}S   \\
			\|z\| =1 \\
	\end{array} } \Phi'(u_{n})z,
\end{equation}
which implies $ \| \Psi'(w_{n}) \|_{S} \leq \|u_{n}\| \|\Phi'(u_{n})\|$.
On the other hand,
	 \begin{equation} \label{A}
	\|\Phi'(u_{n})\| = \sup_{
		\tiny 	\begin{array}{c}
			z  \in T_{w_{n}}S   \\
			z+tw_{n} \neq 0    \\
	\end{array}          }  \frac{\Phi'(u_{n})(z+tw_{n})}{ \|z+tw_{n}\| }
	= \sup_{
		\tiny 		\begin{array}{c}
			z  \in T_{w_{n}}S   \\
			z+tw_{n} \neq 0    \\
	\end{array}   } \frac{\Phi'(u_{n})z}{ \|z+tw_{n}\| },
\end{equation}
where the last equality follows from
$$\Phi'(u_{n})w_{n} = \Phi'(u_{n}) \frac{u_{n}}{\|u_{n}\|} = 0 .$$
Note that, for each $w\in S, z \in T_{w}S$, 
\begin{equation} \nonumber
	\|z\| \leq \|z+tw\| +|t|,
\end{equation}
 \begin{equation} \nonumber
      |t|= |\phi'(w) (z+tw) | \leq  \|\phi'(w)\| \|z+tw\|,
 \end{equation}
where $\phi(w)=\frac{1}{p}\|w\|^{p}$.
Futhermore, by H\"{o}lder inequality,
\begin{equation}\nonumber
	\begin{aligned}
		\phi'(w)v =& \sum_{x,y\in\mathbb{V},x\sim y} \frac{1}{2}|\nabla_{xy} w|^{p-2} \nabla_{xy} w  \nabla_{xy} v + \sum_{ x\in \mathbb{V}} V(x) |w(x)|^{p-2}wv  \\
		\leq &\, \Big(\frac{1}{2}\sum_{x,y\in\mathbb{V},x\sim y} |\nabla_{xy} w |^{p}  \Big)^{\frac{p-1}{p}} \Big(\frac{1}{2}\sum_{x,y\in\mathbb{V},x\sim y} |\nabla_{xy} v|^{p}  \Big) ^{\frac{1}{p}}  
		           + \Big(\sum_{x\in\mathbb{V}} V(x)|w|^{p} \Big)^{\frac{p-1}{p}} \Big(\sum_{x\in\mathbb{V}} V(x)|v|^{p} \Big)^{\frac{1}{p}}\\
		           \leq & \, 2 \|w\|^{p-1} \|v\|
	\end{aligned}
\end{equation}
	for all $v \in T_{w}S$, which implies
	$	\| \phi'(w) \| \leq 2$
and
   \begin{equation} \label{t}
	 \|z\| \leq 3 \|z+tw\|.
   \end{equation}
Plugging \eqref{t} into \eqref{A}, we get 
	\begin{equation}\nonumber
		\| \Phi'(u_{n}) \| \leq 3\sup_{z\in T_{w_{n}}S \backslash \{0\}} \frac{ \Phi'(u_{n})z}{\|z\|}.
	\end{equation}
It follows that by Corollary \ref{corollary3.5},
\begin{equation} \label{``}
	\| \Psi'(w_{n})\|_{S} \leq \|u_{n}\| \| \Phi'(u_{n}) \| \leq 3 \| \Psi'(w_{n}) \|_{S}.
\end{equation}
Futhermore, by  Proposition \ref{3.1 }(ii) and Proposition \ref{coercive}, since $\Phi(u_{n})$ is bounded,
\begin{equation} \nonumber
     \delta \leq \|u_{n}\| \leq C, \quad \forall n\in \mathbb{N}
\end{equation}  
for some $0<\delta <C<\infty$,   proving the result.
\end{proof}
 By $(\ref{``})$ and Proposition \ref{3.1 }(ii), we have the following corollary.
\begin{corollary}
	Let $w \in S$ and $u = m(w) \in \mathcal{N}.$ Then $w\in S$ is a critical point for $\Psi$ whenever $u$ is a critical point of $\Phi$ in $E$.
\end{corollary}

\section{Proof of Theorem \ref{thm1}}
In this section, we prove Theorem \ref{thm1}.
	We first show that
	\begin{equation}\label{c>0}
		c = \inf_{\mathcal{N}} \Phi  \geq \inf_{ S_{\alpha}} \Phi >0
	\end{equation}
	for some $\alpha >0,$ 	where
	$ S_{\alpha} := \{ u\in E: \|u\| = \alpha \}. $
	By $(A_{3})$, 
	$$\sum_{x\in \mathbb{V}} F(x,u)  = o(\|u\|_{p}^{p}) = o(\|u\|^{p})  \quad \text{ as } \|u\| \to 0.$$
Since
$$\Phi(u) = \frac{1}{p}\|u\|^{p} - \sum_{x\in\mathbb{V}} F(x,u),$$
 there exists $\alpha >0$ such that 
	  $ \inf_{u \in S_{\alpha}} \Phi >0$.
By Proposition \ref{3.2}, for each $u\in \mathcal{N}$, there is a unique $t>0$ such that $tu \in S_{\alpha}$, and  then Proposition \ref{3.1 }(i) implies
	$$\Phi(u) \geq \Phi(tu) \geq \inf_{ S_{\alpha}} \Phi .$$
Taking the infimum over $\mathcal{N}$, we obtain $(\ref{c>0})$.
    Note that $\inf_{ S} \Psi = \inf_{ \mathcal{N}} \Phi = c >0 $ and $\Psi \in C^{1}(S)$.
Then, by Corollary \ref{corollary3.5}, for any $u$ attains $\inf_{ \mathcal{N}} \Phi$, $u$ is a critical point of $\Phi$ in $E$. It suffices to prove that $c=\inf_{ \mathcal{N}} \Phi$ is attainable for some $u \in \mathcal{N}$.
     It follows from the Ekeland's variational principle on Banach manifolds that there is a minimizing sequence $\{ w_{n} \} \subset S$, such that
   $$\Psi (w_{n}) \to c,\, \Psi'(w_{n}) \to 0 \quad \text{ as } n\to \infty. $$
Let $u_{n} = m(w_{n}) \in \mathcal{N}$. Then
	\begin{equation}
	\Phi (u_{n}) \to c,\, \Phi'(u_{n}) \to 0 \quad \text{ as } n\to \infty
   \end{equation} 
by Proposition \ref{3.7}. Moreover,
	\begin{equation}\label{K}
		d:= \sup_{n} \|u_{n}\| < \infty
	\end{equation}
as $\Phi$ is coercive on $\mathcal{N}$.
	 Then, by Lemma \ref{embedding}(ii), there exists $u\in E$ such that
	 $$ u_{n} \rightharpoonup u   \quad \text{ and } \quad  u_{n}(x) \to u(x),\quad\forall x\in \mathbb{V}$$
	after passing to a subsequence.
It follows that, for any $w\in C_{c}(\mathbb{V})$,
	\begin{equation}\notag
		\begin{aligned}
			\langle \Phi'(u_{n}),w \rangle 
		=	& \frac{1}{2}\sum_{x,y\mathbb{V},x\sim y} |\nabla_{xy} u_{n}|^{p-2} (\nabla_{xy} u_{n} ) (\nabla_{xy} w )
		+ \sum_{ x\in \mathbb{V}} V(x) |u_{n}(x)|^{p-2} u_{n} w  -\sum_{x\in\mathbb{V}} f(x,u_{n}) w  \\
		\to	& \frac{1}{2} \sum_{x,y\in\mathbb{V},x\sim y} |\nabla_{xy} u|^{p-2} (\nabla_{xy} u ) (\nabla_{xy} w )+ \sum_{ x\in \mathbb{V}}V(x) |u(x)|^{p-2} u w  -\sum_{x\in\mathbb{V}} f(x,u) w  \\
		=	&\, \langle \Phi'(u),w \rangle.
		\end{aligned}
	\end{equation}
Hence, by (\ref{K}),	
\begin{equation}\label{eq17}
	\Phi'(u) = 0.
\end{equation}
Suppose that $u \neq 0$ in $E$, i.e., $u(x_{0}) \neq 0$ for some $x_{0} \in \mathbb{V}$, by $\Phi'(u)u=0$, we get $u\in \mathcal{N}$ a nontrivial critical point for $\Phi$.
 It remains to show $\Phi(u) = c$. 
 Since $\Phi(u_{n}) \to c$ and $ \Phi'(u_{n}) \to 0$,
\begin{displaymath}
	\begin{aligned}
		c &= \lim_{n\to \infty} \big( \Phi(u_{n}) -\frac{1}{p}\Phi'(u_{n})u_{n}\big)
		=\lim_{n\to \infty} \sum_{x\in\mathbb{V}} \Big(\frac{1}{p}f(x,u_{n}) u_{n} - F(x,u_{n}) \Big)  \\
		& \geq \sum_{x\in\mathbb{V}} \Big(\frac{1}{p}f(x,u)u - F(x,u) \Big) \\
		&= \Phi(u), \\
	\end{aligned}
\end{displaymath}
where the last second inequality follows from Lemma \ref{f&F} and Fatou's lemma.
Taking into consideration $\Phi(u) \leq c =\inf_{ \mathcal{N}} \Phi $ and $u \in \mathcal{N}$, we have $\Phi(u) = c$, proving the result. 
Otherwise, $u_{n}(x) \to u(x) = 0$ for each $x\in \mathbb{V}$.
Suppose that $u_{n} \to 0$ in $\ell^{q}(\mathbb{V})$. By Lemma \ref{epsilon}, for any $\epsilon>0$, there is $C_{\epsilon}>0$ such that 
$$ |\sum_{x\in \mathbb{V}} f(x,u_{n})u_{n}  | 
\leq \epsilon \|u_{n}\|_{p}^{p} + C_{\epsilon}\|u_{n}\|_{q}^{q},\quad 
\forall n \in \mathbb{N}. $$
Then, since $\|u_{n}\|_{p}^{p} \leq V_{1}^{-1}\|u\|^{p}$,
\begin{displaymath}
	\begin{aligned}
		 \Phi'(u_{n})u_{n} &= \|u_{n}\|^{p} - \sum_{x\in\mathbb{V}} f(x,u_{n})u_{n}  \\
		                                &\geq(1- \epsilon V_{1}^{-1})\|u_{n}\|^{p} - C_{\epsilon} \|u_{n}\|_{q}^{q}\\
		                                &=\frac{1}{2} \|u_{n}\|^{p} - C_{\frac{V_{1}}{2}} \|u_{n}\|_{q}^{q},\\
	\end{aligned}
\end{displaymath}
where the last equality follows from taking $\epsilon = \frac{V_{1}}{2}$.
We obtain $\|u_{n}\| \to 0 $ since $\Phi'(u_{n}) \to 0$ and $u_{n} \to 0$ in $\ell^{q}(\mathbb{V}) $ as $n\to \infty$. 
This contradicts Proposition \ref{3.1 }(ii) that $\|u_{n}\| \geq \delta >0$ for each $u_{n}\in \mathcal{N}$.
Thus $u_{n} \nrightarrow 0$ in $\ell^{q}(\mathbb{V})$, i.e., 
$$ \varlimsup_{n} \|u_{n}\|_{q} =: \beta >0. $$ 
The interpolation inequality gives 
$$\|u_{n}\|_{q} \leq \|u_{n}\|_{p}^{\frac{p}{q}} \|u_{n}\|_{\infty}^{\frac{p-q}{q}},$$
so that 
$$\varlimsup_{n} \|u_{n}\|_{\infty} \geq (\frac{\beta^{^{q}}}{d^{^{p}}})^{\frac{1}{q-p}} := 2\gamma >0.$$
We have up to a subsequence $|u_{n}(z_{n})| \geq \gamma >0 $ for some $z_{n} \in \mathbb{V}$, $\forall n\in \mathbb{N}$.
Let $\tilde{u}_{n}(\cdot) = u_{n}(\cdot + k_{n}T)$ with $k_{n} = (k_n^{1}, \cdots ,k_{n}^{N}) \in \mathbb{V}$ such that $\{ z_{n} - k_{n}T \}_{n} \subset \Omega $, where $\Omega = [0,T)^{N} \cap \mathbb{V}$ is a bounded set. Then 
$$ \| \tilde{u}_{n} \|_{\ell^{\infty(\Omega)}} \geq |u_{n}(z_{n})| \geq \gamma >0. $$
Since $f(x,u)$ and $V(x)$ are both T-periodic, we have
\begin{equation}\notag
	 \Phi(\tilde{u}_{n}) = \Phi(u_{n}) \to c,\quad 
	\Phi'(\tilde{u}_{n}) = \Phi'(u_{n}) \to 0 \quad \text{ as } n\to \infty.
\end{equation}
Noting that $\{ \tilde{u}_{n} \}$ is bounded as well, we obtain up to a subsequence
$$ \tilde{u}_{n}\rightharpoonup \tilde{u}\quad\text{ and } \quad  \tilde{u}_{n}(x) \to \tilde{u}(x), \forall x\in \mathbb{V},$$
 as $n\to \infty$. The same reasoning as \eqref{eq17} leads to 
 $$\Phi'(\tilde{u}) = \lim_{n} \Phi'(\tilde{u}_{n}) = 0$$
 with $\tilde{u} \neq 0$.
The same argument as before shows that $\Phi(\tilde{u}) = c$.
 This  proves Theorem \ref{thm1}.
 \\ \hspace*{\fill} \\ 
 
\textbf{Acknowledgements.}
The authors would like to thank Minghao Li for valuabe comments and suggestions. B.H. is supported by NSFC, no. 11831004 and Shanghai Science and Technology Program [Project No. 22JC1400100]. W.Xu is supported by Shanghai Science and Technology Program [Project No. 22JC1400100].

\bibliography{ref}
\bibliographystyle{alpha}

\end{document}